\documentclass[12pt,reqno]{amsart}
\usepackage{amssymb,amsmath,amsthm,mathrsfs}
\usepackage{mathtools}
\usepackage{amscd}
\usepackage{enumerate}
\usepackage{subcaption}
\usepackage{graphicx}
\usepackage{color}
\usepackage[dvipsnames]{xcolor}
\usepackage{xcolor}
\numberwithin{equation}{section}
\usepackage{enumitem}
\usepackage[pagebackref=true, colorlinks=true, citecolor=blue]{hyperref}
\usepackage{paralist}
\usepackage[misc]{ifsym}
\usepackage{epsfig} %For pictures: screened artwork should be set up with an 85 or 100 line screen
\usepackage{epstopdf} %This is to transfer .eps figure to .pdf figure; please compile your article using PDFLaTex or PDFTeXify.
\usepackage[colorlinks=true]{hyperref}
\usepackage{mathabx}
\usepackage[normalem]{ulem}
\usepackage{lineno}
\usepackage{mathtools}%                  http://www.ctan.org/pkg/mathtools
\usepackage[tableposition=top]{caption}% http://www.ctan.org/pkg/caption
\usepackage{booktabs,dcolumn}%           http://www.ctan.org/pkg/dcolumn + http://www.ctan.org/pkg/booktabs
\DeclareFontFamily{OT1}{rsfs}{}
\DeclareFontShape{OT1}{rsfs}{n}{it}{<-> rsfs10}{}
\DeclareMathAlphabet{\mathscr}{OT1}{rsfs}{n}{it}

\theoremstyle{plain}

\newtheorem{theorem}{Theorem}[section]

\newtheorem{lemma}{Lemma}[section]

\theoremstyle{definition}

\newtheorem{definition}[theorem]{Definition}
\newtheorem{remark}{Remark}[section]

\def\be{\begin{equation}}
\def\ee{\end{equation}}

\newcommand{\bu}{\mathbf{u}}

\newcommand{\bx}{\mathbf{x}}

\def\bes{\begin{equation*}}
\def\ees{\end{equation*}}
\def\bea{\begin{equation} \begin{aligned}}
\def\eea{\end{aligned} \end{equation}}
\def\beas{\begin{equation*} \begin{aligned}}
\def\eeas{\end{aligned} \end{equation*}}
\def\f{\frac}
\def\p{\partial}

\def\bi{\begin{itemize}}
\def\ei{\end{itemize}}

\newcommand{\ignore}[1]{}

\newcommand\R{\mathbb{R}}

\newcommand{\pa}{\partial}
\newcommand{\na}{\nabla}

\newcommand{\va}{\varepsilon}
\newcommand{\lm}{\lambda}

\newcommand\mcL{\mathcal{L}}

\newcommand{\stkout}[1]{\ifmmode\text{\sout{\ensuremath{#1}}}\else\sout{#1}\fi}
\usepackage{algorithm, algorithmic}
 \let\oldequation\equation
 \let\oldendequation\endequation
 \renewenvironment{equation}
 {\linenomathNonumbers\oldequation}
 {\oldendequation\endlinenomath}
 \let\oldalign\align
 \let\oldendalign\endalign
 \renewenvironment{align}
 {\linenomathNonumbers\oldalign}
 {\oldendalign\endlinenomath}

 \let\oldgather \gather 
 \let\oldendgather\endgather 
 \renewenvironment{gather}
 {\linenomathNonumbers\oldgather}
 {\oldendgather\endlinenomath}

\title[Darcy-Boussinesq System in Layered Porous Media]
{Well-posedness and Regularity of the Darcy-Boussinesq System in Layered Porous Media} % Only the first word and proper nouns should be capitalized.

\author[Yining Cao, Weisheng Niu and Xiaoming Wang]{}

\thanks{This work is supported by NNSF of China (11971031)  (Niu) and NNSF of Anhi Province (2108085Y01)(Niu),  NSFC 12271237 (Wang) and the Havener Endowment (Wang).}

\thanks{$^*$Corresponding author: Xiaoming Wang}

\begin{document}

%\linenumbers  
 \maketitle
\centerline{\scshape
Yining Cao $^{{\href{mailto:yinicao.math1@gmail.com}{\textrm{\Letter}}}1}$,
  Weisheng Niu $^{{\href{mailto:niuwsh@ahu.edu.cn}{\textrm{\Letter}}}2}$,
 and Xiaoming Wang $^{{\href{mailto:wxm.math@outlook.com}{\textrm{\Letter}}}3, 1}$}
\medskip

{\footnotesize
 \centerline{$^1$Department of Mathematics, Southern University of Science and Technology, China}}

%\medskip

{\footnotesize
 \centerline{$^2$ School of Mathematical Science, Anhui University,
 CHINA}
  \centerline{$^3$ Department of Mathematics and Statistics, Missouri University of Science and Technology,  USA
 }
}

\begin{abstract}
We investigate the Darcy-Boussinesq model for convection in layered porous media. In particular, we establish the well-posedness of the model in two and three spatial dimension, and derive the regularity of the solutions in a novel piecewise $H^2$ space.
\end{abstract}

%S1
\section{Introduction}
%{\color{red}To be inserted by Wang.}
Convection, i.e., the fluid motion due to differential heating,  is a fascinating topic. It also serves as a paradigm for a plethora of nonlinear phenomena. See for instance the classical treatise by Nobel Laureate Chandresekhar \cite{chandrasekhar1961oxford} as well as the book by Drazin and Reid \cite{drazin1981hydrodynamic}.
Convection in porous media, highly relevant to geophysical applications and many engineering problems, has been the focus of many researchers. The treatise by Nield and Bejan \cite{nield2017convection} is an excellent survey of convection in porous media from the physical/geophysical side. There are also several mathematical works in this area by Fabrie, Nicholaenko, Ly, Titi, Oliver, Doering, Constantin, Otero et al that cover rigorous bound on the Nusselt number, well-posedness of the system and their long time behavior etc \cite{doering1998jfm, fabrie1986aam, fabrie1996aa, ly1999jns, oliver2000gevrey, otero2004jfm}.  All these works, except a few section from the book of Nield and Bejan, deal with the case when the porous media is essentially homogeneous in the sense that the permeability and other parameters are either constants or are nice smooth functions of the spatial variable.

On the other hand, many natural and engineered porous media are of layered structure in the sense that the permeability and other physical parameters are piecewise constants. Such layered porous media, related to the technology of underground carbon dioxide (CO2) sequestration has received quite some attention from the fluid and environmental community \cite{bickle2007modelling, huppert2014ar, hewitt2014jfm, hewitt2020jfm, hewitt2022jfm, mckibbin1980jfm, mckibbin1981heat, mckibbin1983thermal, sahu2017tansp, salibindla2018jfm, wooding1997convection}. The review paper by Huppert and Neufeld \cite{huppert2014ar} provides an excellent survey.

In the current work, we investigate the well-posedness of convection in layered porous media.
Due to the layered nature of the porous media, the solution is no longer smooth, as oppose to earlier works on the homogeneous case \cite{fabrie1986aam, fabrie1996aa, ly1999jns, oliver2000gevrey}. We show that the solution belongs to a piecewise smooth function space with appropriate interfacial boundary conditions.
We believe that this is the first rigorous mathematical work on convection in layered porous media.

The rest of the paper is organized as follows. We provide the setup of the problem as well as some preliminaries  in section 2. The existence of global weak solution is presented in section 3. Regularity and uniqueness of solutions are presented in sections 4 and 5 for the two and three dimensional cases respectively.  Concluding remarks are offered in section 6.

%S2
\section{Formulation of the Problem}
%S 2.1
\subsection{Physical Model}
%In the two-dimensional case, we consider the idealized layered domain $\Omega=(0,L) \times(-H, 0)$ with constants $z_j, 0 \leq j \leq \ell $ satisfying
For $d=2,3$, we consider the idealized layered domain $\Omega=(0,L)^{d-1} \times(-H, 0)$ with constants $z_j, 0 \leq j \leq \ell $ satisfying
\begin{align}
-H \equiv  z_l<\cdots<z_{0} \equiv 0 . \notag
\end{align}
The `layers'' or "strips",  i.e., the $\Omega_j$s are defined as follows:
\begin{align}
\Omega_j=\left\{\textbf{x}=(x,z) \in \Omega \mid z_{j}<z<z_{j-1}\right\}, \quad 1 \leq j \leq   \ell .\notag
\end{align}
%The generalization to the three-dimensional case is straightforward.
For convection in this layered domain, the governing equations in $\Omega$ are the following Darcy-Boussinesq system (with the usual Boussinesq approximation)\cite{nield2017convection}
\begin{align}
\operatorname{div}(\mathbf{u})=0, \quad \mathbf{u}=\left(u_1,\cdots, u_d\right), \label{1.1}\\
\mathbf{u}=-\frac{K}{\mu}\left(\nabla P+\rho_0(1+\alpha \phi) g \mathbf{e}_z \right),\label{1.2} \\
b \frac{\partial \phi}{\partial t}+\mathbf{u} \cdot \nabla \phi-\operatorname{div}(b D \nabla \phi)=0.\label{1.3}
\end{align}
Here  $\mathbf{u}, \phi$ and $P$ are the unknown fluid velocity, concentration, and pressure, respectively; $\rho_0, \alpha,  \mu, g$ are the constant reference fluid density,   constant expansion coefficient,  constant dynamic viscosity, and the gravity acceleration constant, respectively; and $\mathbf{e}_z $ stands for the unit vector in the $z$ direction. In addition,
$K, b, D$ represent the permeability, porosity, and diffusivity coefficients respectively, which are assumed to be constant within each strip/layer  $\Omega_j$.
Namely,
\begin{align}
K=K(\bx)=K_j, b=b(\bx)=b_j, D=D(\bx)=D_j, \quad \bx \in \Omega_j, \quad 1 \leq j\leq \ell,   \notag
\end{align}
for a set of constants $\left\{K_j, b_j, D_j \right\}_{j=1}^{\ell}$.

System \eqref{1.1}-\eqref{1.3} modeled convection in a layered porous media where each layer is of different permeability/porosity/diffusivity. On the interfaces $z= z_j$,  we assume
\begin{eqnarray}\label{IBC_up}
    \mathbf{u}\cdot \mathbf{e}_z ,\,\, P \text{ are continuous at } z= z_j, \, 1 \leq j \leq  \ell-1 ,
    \\
  \phi, \,\,  bD \frac{\partial \phi}{\partial z}  \text{ are continuous at } z= z_j, \, 1 \leq j \leq  \ell-1.
  \label{IBC}
\end{eqnarray}
%for $0 \leq i \leq\ell$.
Interfacial boundary condition \eqref{IBC} implies that the solution $\phi$ cannot be smooth over the whole domain in general unless $bD$ is a constant.
This is one of the main challenges of this problem.

System \eqref{1.1}-\eqref{1.3} is supplemented with the initial condition
\be\label{IC}
\phi( x,z , 0)=\phi_0(x,z),
\ee
and the boundary conditions
\begin{equation}\label{BC}
\begin{split}
\mathbf{u}\cdot e_z \left(x, 0 ; t\right)=\mathbf{u}\cdot e_z \left(x,-H ; t\right)=0, \\
%&\quad \text { and } \quad \mathbf{u}\left(0, z ; t\right)=\mathbf{u}\left(1, z ; t\right)\\
\phi\mid_{z=0}=C_0, \quad \phi\mid_{z=-H}=C_1,
%, &\quad \text { and } \quad \phi\left(0, z ; t\right)=\phi\left(1, z ; t\right).
\end{split}
\end{equation}
together with periodicity in the horizontal direction(s) $ x$ ($x=(x_1,x_2)$ in the three-dimensional case).

\begin{remark}
    By a change of variable $\widetilde{\phi}(x,z, t)=\phi(x,z, t)+\frac{1}{H}\left(C_1 z-C_0\left(H+z\right)\right)$, the boundary condition \eqref{BC}$_2$ for $\phi$ can be homogenized, and the  extra terms involving $C_0, C_1$ appearing in the new set of equations similar to those \eqref{1.3} are lower order terms and easy to handle. Hereafter for simplicity,   and without loss of generality, we assume that $C_0=C_1=0$ . Thus,
\eqref{BC}$_2$ becomes
    \begin{equation}\label{hBC}
       \phi\mid_{z=0}=0,  \quad \phi\mid_{z=-H}=0.
        %, \quad \text { and } \quad \phi\left(0, z ; t\right)=\phi\left(1, z ; t\right).
    \end{equation}
\end{remark}

\begin{remark}
By setting $\widetilde{P}=P-\rho_0  g z$, and omitting $\sim$ for simplicity, we
may rewrite \eqref{1.2} as
\begin{align}
\mathbf{u}=-\frac{K}{\mu}\left(\nabla P+ \alpha \rho_0 g \phi \mathbf{e}_z \right). \label{1.2'}
\end{align}
Notice that the interfacial conditions \eqref{IBC_up} remain unchanged under the change of variable.
We will adopt this new formulation  hereafter.
\end{remark}

% S2.2 
\subsection{Weak formulation}
Let $L^p(\Omega)$ and $H^k(\Omega)$ denote the usual $L^p$-Lebesgue space of integrable functions and $H^k$ Sobolev spaces that are periodic in the horizontal direction(s), respectively, for $1 \leq p \leq \infty$ and $k \in \mathbb{R}$. The inner product in $L^2(\Omega)$ will be denoted by $(\cdot , \cdot)$. Let
 \begin{align}
\begin{gathered} \notag
\mathcal{V}:=\left\{\phi \in C(\bar{\Omega}) \text{ and }\phi\mid_{\Omega_j}\in C^{\infty}(  \bar{\Omega}_j  ),\, \, 1\leq  j \leq l  \mid \phi \text { satisfies }\eqref{IBC},   \eqref{hBC}\right\}, \\
V:=\text { Closure of } \mathcal{V} \text { in the } H^1-\text {norm},\\
H:=\text { Closure of } \mathcal{V} \text { in the } L^2-\text {norm},
\end{gathered}
 \end{align}
and let us denote the $L^2$-norm of $H$ by $\|\cdot\|_H$, and the norm of $V$ by $\|\cdot\|_V$. The inner product of $H$ is exactly the inner product of $L^2(\Omega)$. Notice that due to the boundary conditions \eqref{hBC}, the Poincar\'{e} inequality implies that the $V$-norm and the $H^1$-Sobolev norm are equivalent and thus,  when combined with the lower and upper bounds on $b$ and $D$, we can define
% V norm
 \begin{equation}\label{V-norm}
 \|\phi\|_V=\|\sqrt{bD}\nabla \phi\|_{L^2(\Omega)}.
 \end{equation}
We also recognize that $V=H^1_{0,per}(\Omega)$, the subspace of $H^1(\Omega)$ that vanishes at $z=0, -H$ and periodic in the horizontal direction.
We denote the dual space of $V$ by $V^* $ with norm $\|\cdot\|_{V^{*}}$. The symbol $\langle\cdot, \cdot\rangle$ will stand for the duality product between $V$ and $V^* $.

Let us also define
\begin{align}
\begin{gathered}\notag
\tilde{\mathcal{V}}:=\left\{\mathbf{u} \in C(\bar{\Omega})^d \text{ and }\mathbf{u}\mid_{{\Omega}_j}\in C^{\infty}(\bar{\Omega}_j)^d,\, \, 1 \leq j \leq  \ell  \mid \mathbf{u}\text { satisfies }\eqref{1.1},\eqref{IBC_up}, \eqref{BC}_1\right\} \\
%\mathbf{V}:=\text { Closure of } \tilde{\mathcal{V}} \text { in the } H^1-\text {norm } \\
\mathbf{H}:=\text { Closure of } \tilde{\mathcal{V}} \text { in the } L^2-\text {norm. }
\end{gathered}
\end{align}
%Let $\mathrm{P}_\sigma:\left(L^2(\Omega)\right)^2 \rightarrow \mathbf{H}$ be the Helmholtz-Leray orthogonal projection. By applying $\mathrm{P}_\sigma$ to \eqref{1.2} and using \eqref{1.1}, we have
%\begin{equation}\label{LerayP}
%\mathbf{u}+\f{\rho_0 g K}{\mu}\mathrm{P}_\sigma(\phi)=0.
%\end{equation}

By applying the divergence operator to \eqref{1.2'} within each subdomain $\Omega_j$, $1 \leq j \leq  \ell $, we obtain that  
 \begin{equation} \notag
 -\int_{\Omega_j}\frac{K}{\mu} (\nabla  P  -  \alpha \rho_0 g \phi \mathbf{e}_z  )\cdot \nabla q  d\bx  +  \int_{\pa\Omega_j}\frac{K}{\mu} (\nabla  P  -  \alpha \rho_0 g \phi \mathbf{e}_z  )\cdot n_j   q  d\sigma=0
 \end{equation} for any function $q\in H^1(\Omega)$ which is L-periodic in $x$, where $n_j$ denotes the unit outward normal to the boundary $\pa \Omega_j.$
 By summation over $j$,  and using the boundary conditions \eqref{IBC_up} and  \eqref{BC}$_1$, we get 
 \begin{align}\notag
 & \int_{\Omega}\frac{K}{\mu} (\nabla  P  -  \alpha \rho_0 g \phi \mathbf{e}_z  )\cdot \nabla q  d\bx  =0.
 \end{align}
This, together with   \eqref{BC}$_1$, \eqref{hBC} and the periodicity of $\mathbf{u}, \phi$ in the horizontal direction(s), implies the equations for $P$ in $\Omega$:
\begin{equation}\label{eqP}
\begin{cases}
-\text{div} \Big( \frac{K}{\mu} \nabla P\Big) =     \text{div} \Big(  \frac{ \alpha \rho_0 g  K}{\mu} \phi \mathbf{e}_z  \Big) \quad \text{ in } \Omega,   \\
 \frac{\partial P}{\partial z }(x,0) =  \frac{\partial P}{\partial z }(x,-H)=0      ,
 %\\
%P(0,z)=P(1,z) .
\end{cases}
\end{equation}
together with periodicity in the horizontal direction.
By the Lax-Milgram theorem, the above equation admits a unique, up to a constant, solution in $V$ for any $\phi\in H$.
This means that for any function $q\in H^1(\Omega)$ which is L-periodic in $x$, it holds that
\begin{align}
\int_\Omega \frac{K}{\mu} \nabla  P\cdot \nabla q d\bx =   \int_\Omega    \frac{ \alpha \rho_0 g  K}{\mu} \phi \mathbf{e}_z \cdot \nabla q  d\bx .
\end{align}
 Moreover, from the $W^{1,p}$ estimates for the elliptic systems with co-normal boundary conditions and variably partially small BMO coefficients \cite{dongli21jfa}, we have
\begin{equation}\label{p-phi}
    \|P \|_{W^{1,p}(\Omega)} \leq C \| \phi \|_{L^p(\Omega)}, \,\, \text{ for any } 1 < p <\infty.
\end{equation}
In view of  \eqref{1.2}, we have
\begin{equation}\label{u-phi-lp}
     \|\mathbf{u} \|_{L^p(\Omega)} \leq C (\|P \|_{W^{1,p}(\Omega)} + \| \phi \|_{L^p(\Omega)} ) \leq   \| \phi \|_{L^p(\Omega)}, \,\, 1 < p <\infty.
\end{equation}
Thanks to \eqref{1.2} and  \eqref{eqP}, the velocity $\mathbf{u}$ is a function of the concentration $\phi$.  As a result, \eqref{1.3} can be viewed as an equation of $\phi$ only.

Multiplying \eqref{1.3} by a test function $\psi \in V$, and integrating over $\Omega$,  we get
\ignore{
\begin{equation}
    \langle b\p_t \phi,\psi\rangle_{\Omega_j} +   \left( bD\nabla \phi, \nabla \psi\right)_{\Omega_j}+\int_{\Omega_j}\left(\mathbf{u}\cdot \nabla \phi \right)\psi d\bx- \int_{\p \Omega_j}bD\nabla \phi \cdot \mathbf{n}_j \psi d\sigma=0,
\end{equation}
where $\mathbf{n}_j$ is the outward unite normal to $\partial \Omega_j$. Denote the upper and lower boundary of $\Omega_j$ as $\Gamma_j^{+}$ and $\Gamma_j^{-}$, respectively. Since the function $\mathbf{u}, \phi$ are periodic with respect to the $x$ variable,
$$
\int_{\partial \Omega_j} b D \nabla \phi \cdot \mathbf{n}_j \psi d \sigma=\int_{\Gamma_j^{+}} b D \nabla \phi \cdot \mathbf{n}_j \psi d \sigma+\int_{\Gamma_j^{-}} b D \nabla \phi \cdot \mathbf{n}_j \psi d \sigma .
$$
By the continuity of vertical flux across the interface
$$
\int_{\Gamma_{j+1}^{+}} b D \nabla \phi \cdot \mathbf{n}_{j+1} \psi d \sigma=-\int_{\Gamma_j^{-}} b D \nabla \phi \cdot \mathbf{n}_j \psi d \sigma.
$$
Summing over $j$, $j=  1,\cdots, \ell $,  we deduce that
}
\begin{equation}\label{wf}
    \langle b\p_t \phi,\psi\rangle +   \left( bD\nabla \phi, \nabla \psi\right)+\int_{\Omega}\left(\mathbf{u}\cdot \nabla \phi \right)\psi d\bx=0, \forall \psi \in V,
\end{equation}
where we have performed integration by parts and utilized the boundary and interfacial conditions.

We then define the bilinear forms $A(\phi,\psi )$  and $ B(\mathbf{u}, \phi,\psi )$  for some $\mathbf{u}\in \mathbf{H}\cap L^p(\Omega)^d$, with a $p>2$:
\begin{align} \notag
A(\phi,\psi )&=  \left( bD\nabla \phi, \nabla \psi\right), \forall \phi, \psi \in V , \\
B(\mathbf{u}, \phi,\psi )&= \int_{\Omega}\left(\mathbf{u}\cdot \nabla \phi \right)\psi d\bx, \forall \phi, \psi \in V. \notag
\end{align}
The weak solution to the system \eqref{1.1}-\eqref{BC} is defined as follows.
\begin{definition}\label{def1}
Let $\phi_0 \in L^2(\Omega)$  be given, and let $T>0$. A weak solution of \eqref{1.1}-\eqref{1.3}, subject to the boundary conditions \eqref{IBC_up}, \eqref{IBC}, \eqref{BC}$_1$, \eqref{hBC} together with the periodic conditions in the horizontal direction, and the initial condition \eqref{IC} on the interval $[0,T]$ is a triple $(\mathbf{u},\phi, P)$, satisfying
\begin{align}
\phi \in L^2(0, T ; V) \cap  L^{\infty}\left(0, T ; L^2(\Omega)\right) \text{ and }{\partial \phi}/{\partial t} \in L^2\left(0, T ; {V^* }\right),   \notag
\end{align}
and
\begin{align}\label{def-eq}
\left(b \phi(t_2), \psi\right)-\left(b \phi(t_1), \psi\right)+\int_{t_1}^{t_2}A( \phi, \psi) dt+\int_{t_1}^{t_2}B(\mathbf{u},  \phi,\psi) d t=0,\end{align}
$\forall  \psi \in V, t_1, t_2 \in[0,T]$,  and  $\phi(0)=\phi_0 $ in $L^2(\Omega)$,
where $\mathbf{u} \in L^2\left(0, T ; L^p(\Omega)\right)$ and $P \in  L^2\left(0, T ; W^{1,p}(\Omega)\right) $, with $2 \le p <\infty$ if $d=2$, and $p=6$ if $d=3$,  are given by Darcy's law, i.e., \eqref{1.2} and \eqref{eqP}, respectively.

\end{definition}

% S2.3
\subsection{Natural space}

Since the physical set-up of the problem is different from the classical homogeneous media setting, the natural space for the solution, which is associated with the behavior of the principal differential operator of the system,  is different from the classical setting. Therefore, we need to investigate the behavior of the principal differential operator of the system, i.e., 
  $\mcL= -\text{div}(bD \na)$, subject to the boundary conditions \eqref{IBC}, \eqref{BC}$_1$ and \eqref{hBC}. 
  
 %We equip $V$ with the equivalent norm of
 % V-norm
 %\begin{equation}\label{V-norm}
 % \|\phi\|_V^2 = \int_\Omega bD|\nabla\phi|^2.
 % \end{equation}
 Define 
 % solution space
  \begin{equation}\label{W}
  W= \{\varphi\in V:  \pa_{x} \varphi \in H^1(\Omega), ~ bD \pa_{z} \varphi  \in H^1(\Omega) \},
 \end{equation}
endowed  with norm
% W norm
\begin{equation}\label{W-norm}
 \|\varphi\|^2_{W}= \|\varphi \|^2_V+ \|\pa_{x}\varphi \|^2_{H^1(\Omega)} + \|bD\pa_{z}\varphi \|^2_{H^1(\Omega)}.
 \end{equation}
 Notice that $W$ is not twice weakly differentiable. In fact, the functions in $W$ are piecewise twice differentiable, but the vertical ($z$) derivative is discontinuous at each interface between two neighboring layers in general. This space is natural since it is the natural space of the eigenfunctions of the principal linear operator $\mathcal{L}$ as we shall demonstrate below.
 The discontinuity of the derivative implies that the classical method can not be applied directly. 
 
 We first show that $W$ is associated with the eigenfunctions of $\mathcal{L}$.
 % lemma 2.1
 \begin{lemma}{eigenfunctions}
 The operator  $\mcL= -\text{div}(bD \na)$ is self-adjoint, and it possesses a set of eigenfunctions $\{w_k\}_{k=1}^\infty\subset W$ which forms an orthonormal basis in $H$, and an orthogonal basis in $V$.
 In addition, the eigenfunctions are smooth in each layer $\Omega_j$
 \end{lemma}
 \begin{proof}
   Note that $\mcL$ is a self-adjoint positive operator, as
 \begin{align}
 \langle\mcL \varphi, \psi\rangle =\int_\Omega dD \na \varphi \na \psi\, d\bx.
 \end{align}
 Thanks to Lax-Milgram theorem, for any $f\in H$ there exists a unique solution $\varphi\in  V$ such that
 $$ \int_\Omega dD \na \varphi \cdot\na \psi d\bx =\int_\Omega f \psi d\bx  \quad \text{ for any }   \psi \in V.$$
 So  $\mcL^{-1}: H\rightarrow H$ is a compact self-adjoint positive operator. As a result, $ \mcL $ admits a sequence of eigenvalues $\{\lm_k\}_{k=1}^\infty$, where $\lm_k\rightarrow \infty$  as $k \rightarrow \infty.$ Moreover,  there exists an orthonormal basis $\{w_k\}_{k=1}^\infty$ of $H$, such that $w_k \in V$ is an eigenfunction corresponding to $\lm_k$:
  \begin{align}\label{eigenequ}
 \mcL  w_k = \lm_k w_k,  \quad   k=1,2,\cdots.
 \end{align}
 
  This implies that  $\{w_k\}_{k=1}^\infty$ also form an orthogonal basis for $V$, a factor that we will utilize in the sequel.

  The eigenfunctions enjoy additional regularity. We first show that regularity in the horizontal direction.
 By taking the derivative of \eqref{eigenequ} with respect to $x$, and then multiplying the resulting equation by $\pa_{x} w_k$ and integrating by parts,  we deduce that
 \begin{align}
 \int_{\Omega_j} b D|\nabla \pa_{x} w_k|^2 d\bx
 = \int_{\partial \Omega_j} b D \nabla \pa_{x}w_k \cdot n_j \pa_{x}w_k  d \sigma+ \int_{\Omega_j} \lm_k  |\pa_{x}w_k|^2  d\bx . \notag
 \end{align}
 By summation and using the boundary conditions  \eqref{IBC}  and \eqref{hBC},  we get
    \begin{align}
 \int_{\Omega} b D |\nabla \pa_{x} w_k|^2 d\bx
\leq    \int_{\Omega} \lm_k  |\na w_k|^2  d\bx .
 \end{align}
Since $ \pa_{x} w_k $ satisfies the boundary and interfacial conditions \eqref{IBC}  and \eqref{hBC}, we have $\pa_{x} w_k \in V\subseteq H^1(\Omega).$  On the other hand,  thanks to \eqref{eigenequ}
\begin{align}
\pa_{z}(bD \pa_{z} w_k)= -bD \pa_{x}(\pa_{x} w_k) -\lm_k w_k, \notag
 \end{align}
 which, together with the fact that  $\pa_{x} w_k  \in H^1(\Omega),$  implies that $bD \pa_{z} w_k \in H^1(\Omega).$  Therefore the eigenfunctions  actually belong to the natural space $W$ \eqref{W}.

  Moreover, since $b, D$ are constant in each $\Omega_j$, the eigenfunctions are piece-wise smooth. Indeed, assume that $w_k, \partial_{x}w_k\in H^m(\Omega_j)$ for some $m\geq 1$. By \eqref{eigenequ}, we have
  \begin{align}
 b_jD_j \frac{\pa^{m+1}w_k}{\pa z^{m+1}} =- b_jD_j \pa_{x}\frac{\pa^{m+1}w_k}{\pa z^{m-1}\pa x^2 } -\lm_k \frac{\pa^{m-1}w_k}{\pa z^{m-1}}, \notag
 \end{align}
which implies that $w_k\in H^{m+1}(\Omega_j)$. By induction, we know that $w_k \in C^\infty(\Omega_j)$.
\end{proof}

The next lemma states that the space $W$ has an equivalent norm. % works almost like $H^2$ 
%%%%% lemma 2.2
\begin{lemma}\label{remark-norm-equi}
There is an equivalent norm on $W$ given by $\|\mathcal{L}\phi\|_{L^2(\Omega)}$, i.e., there exists a $C>0$ such that
\begin{align}\label{remark-norm-equi-00}
\|\varphi\|_{W} \leq C \|\mathcal{L} \phi\|_{L^2(\Omega)}, \quad \forall\phi\in W.\end{align} 
Moreover, the norm $\|\phi\|_{W} $ is equivalent to  $\|\mathcal{L} \phi\|_{L^2(\Omega)}.$

\end{lemma}
\begin{proof}
It is easy to see that $\|\mathcal{L} \phi\|_{L^2(\Omega)}\le C \|\phi\|_W \ \forall \phi\in W$.
For the opposite inequality, we assume that
$$
\mathcal{L} \phi =f \quad \text{ in }  \Omega,$$
and  $\phi$ satisfies \eqref{IBC}, \eqref{hBC}, together with the  periodicity in the horizontal direction(s). By the Lax-Milgram's theorem,  there exists a unique $\phi \in V$ such that
\begin{align}\label{remark-norm-equi-0}
\|\phi \|_{H^1(\Omega)} \leq C \|f\|_{L^2(\Omega)}.
\end{align}
By differentiating the equation with respect to the $x$ variable, we get $$
\mathcal{L} (\pa_{x}\phi) = \pa_{x}f  \quad \text{in  } \Omega,$$ which implies that
\begin{align}\label{remark-norm-equi-1}
\|\pa_{x}\phi \|_{H^1(\Omega)} \leq C \|f\|_{L^2(\Omega)}.
\end{align}
 Since
$-\pa_{z}(bD \pa_{z} \phi)= bD \pa_{x}(\pa_{x} \phi) +f,$ we deduce that
\begin{align}\label{remark-norm-equi-2}
  \|\pa_{z}(bD \pa_{z} \phi)\|_{L^2(\Omega)}\leq \| bD \pa_{x}(\pa_{x} \phi)\|_{L^2(\Omega)} + \|f\|_{L^2(\Omega)}\leq C  \|f\|_{L^2(\Omega)}.
\end{align}
On the other hand,  by the equation
\begin{align}
  \|\pa_{x}(bD \pa_{z} \phi)\|_{L^2(\Omega)}=\| bD \pa_{z}(\pa_{x} \phi)\|_{L^2(\Omega)}  \leq C\|\pa_{x}\phi \|_{H^1(\Omega)} \leq C  \|f\|_{L^2(\Omega)}, \notag
\end{align}
Combining this with \eqref{remark-norm-equi-0},  \eqref{remark-norm-equi-1} and \eqref{remark-norm-equi-2}, we derive \eqref{remark-norm-equi-00}.
%\end{lemma}
\end{proof}

The next lemma indicates that the natural space $W$ is very similar to $H^2$ in terms of some of the commonly utilized Sobolev imbeddings.
%lemma 2.3
\begin{lemma} \label{remark-imdedding}
Moreover, $W$ is similar to $H^2$ in the sense that the following inequalities hold
 \begin{align} \label{remark-imdedding-1}
   \begin{split}\| \na \varphi \|_{L^p(\Omega)}   \leq C  \|\varphi\|_{W} ~ \text{ for }  d=2, \quad \text{ and } \quad
    \| \na \varphi \|_{L^6(\Omega)}   \leq C  \|\varphi\|_{W} ~\text{ for }  d=3.
    \end{split}
    \end{align}
and 
 \begin{equation}\label{phi-infinity}
  \|   \varphi \|_{L^\infty(\Omega)}   \leq C  \|\varphi\|^\frac12\|\varphi\|_{W}^\frac12 \quad \text{ for }  d=2, 
  \quad\text{ and}\quad
  \|   \varphi \|_{L^\infty(\Omega)}   \leq C  \|\varphi\|^\frac14\|\varphi\|_{W}^\frac34 \quad \text{ for }  d=3.
  \end{equation}
  \end{lemma}
  \begin{proof}
 
% remark 3.2

 Note that
 $$\|\varphi\|^2_{W}= \sum_{j=1}^\ell  \left(\|bD\nabla\varphi \|^2_{L^2(\Omega_j) }+ \|\pa_{x}\varphi \|^2_{H^1(\Omega_j)} + \|bD\pa_{z}\varphi \|^2_{H^1(\Omega_j)}\right).$$
 By standard Sobolev imbedding, we have  on each strip (each $j$)
  \begin{align} \notag \begin{split}\| \na \varphi \|^2_{L^p(\Omega_j)} &\leq C \big(\|  \varphi \|^2_{L^2(\Omega_j)} + \| \na^2 \varphi \|^2_{L^2(\Omega_j)}\big)\\
  & \leq C \left(\|\varphi \|^2_{H^1(\Omega_j) }+ \|\pa_{x}\varphi \|^2_{H^1(\Omega_j)} + \|bD\pa_{z}\varphi \|^2_{H^1(\Omega_j)}\right) \end{split}
  \end{align}   for any $ 1< p<\infty$ in the case $d=2$\footnote{The constant $C$ may depend on the exponent $p$.},
  and
  \begin{align} \notag
  \begin{split}\| \na \varphi \|^2_{L^6(\Omega_j)} &\leq C \big(\|  \varphi \|^2_{L^2(\Omega_j)} + \| \na^2 \varphi \|^2_{L^2(\Omega_j)}\big)\\
  & \leq C \left(\|\varphi \|^2_{H^1(\Omega_j) }+ \|\pa_{x}\varphi \|^2_{H^1(\Omega_j)} + \|bD\pa_{z}\varphi \|^2_{H^1(\Omega_j)}\right) \end{split}
  \end{align}  
   in the case $d=3$.

  Summing over $j$, we derive that, in the two dimensional case, $\forall p\in [2,\infty)$,
   \begin{equation} \label{remark-imdedding-1}
  \| \na \varphi \|_{L^p(\Omega)}   \leq C  \|\varphi\|_{W} ,
  \end{equation} 
  and in the three dimensional case
  \begin{equation}
    \| \na \varphi \|_{L^6(\Omega)}   \leq C  \|\varphi\|_{W}.
    \end{equation}
 Likewise, we have the imbedding
 \eqref{phi-infinity}, and 
 \begin{equation}
  \|   \varphi \|_{L^\infty(\Omega)}   \leq C  \|\varphi\|_{W} \quad \text{ for }  d=2, 3.
  \end{equation}
  %which by interpolation yields,
%   \begin{align}
%   \begin{split}\| \na \varphi \|_{L^p(\Omega)}   \leq C  \|\varphi\|_{W} \quad \text{ if }  d=2,\\
%    \| \na \varphi \|_{L^6(\Omega)}   \leq C  \|\varphi\|_{W} \quad \text{ if }  d=3,
%    \end{split}
%  \end{align}
\end{proof}

%%%%% S3
\section{Global Existence of Weak Solution}

This section is devoted to the global existence of weak solutions to problem \eqref{1.1}-\eqref{1.3} subject to the boundary conditions \eqref{IBC_up}, \eqref{IBC}, \eqref{BC}$_1$, \eqref{hBC} together with the periodic conditions in the horizontal direction(s), and the initial condition \eqref{IC}.

%{\color{red} via eigenfunctions of the 2nd order operator $\nabla\cdot(bD\nabla)$, truncation satisfies the same interfacial and boundary conditions for u, P, and $\phi$. Notice that the eigenfunctions of $\nabla\cdot(bD\nabla)$ satisfy the interfacial conditions \eqref{IBC} in the classical sense.  Hence, any finite linear combination of eigenfunctions would satisfy the interfacial boundary conditions \eqref{IBC} in the classical sense.
%The vertical regularity proof would be a little bit tricky since the test function we used is not in the test function space.  We could consider the case with constant porosity first.}

% main existence theorem
\begin{theorem}
Assume that $\phi_0 \in L^2(\Omega)$. The system \eqref{1.1}-\eqref{1.3} subject to the boundary and interfacial conditions \eqref{IBC_up}, \eqref{IBC}, \eqref{BC}$_1$, \eqref{hBC} together with the periodic conditions in the horizontal direction(s), and the initial condition \eqref{IC} admits a global weak solution $(\textbf{u},\phi, P)$ in the sense of Definition 2.1 if the porosity $b$ is a constant.
\end{theorem}

\begin{proof}
We prove the existence of solutions via the standard Galerkin approximation utilizing the eigenfunctions of $\mathcal{L}$ studies in the previous section.

Consider the Galerkin approximation system
\begin{align}
&\operatorname{div}(\mathbf{u}_n)=0,   \label{1.1n}\\
&\mathbf{u}_n=- \frac{K}{\mu}\left(\nabla P_n+ \alpha \rho_0 g \phi_n \mathbf{e}_z \right)  ,\label{1.2n} \\
&b \frac{\partial \phi_n}{\partial t}+ Q_n(\mathbf{u_n} \cdot \nabla \phi_n)-\operatorname{div}(b D \nabla \phi_n)=0, \label{1.3n}
\end{align} with the initial condition
\be\label{ICn}
\phi_n(x, 0)=Q_n \phi_0(x),
\ee
where  $ u_n,\phi_n, P_n$ satisfy the boundary conditions \eqref{IBC_up}, \eqref{IBC}, \eqref{BC}$_1$ and  \eqref{hBC},  and $Q_n$ is the projection from $V$ onto the space $V_n= span\{w_1,\cdots, w_n \}$. We find the solution $\phi_n$ of the form $\phi_n=\sum_{k=1}^n c_k(t)w_k$.  Let $P_n$ be given by \eqref{eqP},  and $\bu_n$ given by \eqref{1.2'} with data $P_n, \phi_n$.   Since $\bu_n$ and $P_n$ are functions of $\phi_n$, and depend linearly on $c_k(t)$. The ordinary differential system \eqref{1.3n} admits a unique local in time solution $c_k(t), k=1,\cdots, n$. Multiply the equation with  $c_k(t) w_k$, sum over $k$ form $1$ to $n$, and utilize integration by parts together with the boundary conditions for $w_k$, we deduce that
\begin{align}
 \frac{d}{d t} \int_{\Omega} b|\phi_n|^2+\int_{\Omega} b D|\nabla \phi_n|^2 =0, \notag
 \end{align}
which implies that
\begin{align}\label{galerkin-eti-1}
\sup _{0 \leq t \leq T} \int_{\Omega}|\phi_n(t)|^2+\int_0^T \int_{\Omega} D|\nabla \phi_n|^2 \leq C \int_{\Omega}|\phi_0|^2\quad \forall n.
\end{align}
Hence,
\begin{equation}
  \phi_n \in L^\infty(0, T; H)\cap L^2(0,T; V),
 \end{equation}
 with bounds uniform in $n$ (independent of $n$).

We distinguish the case $d=2$ and $d=3$ to complete the remaining analysis.\\
\noindent\textbf{The case $d=2$.}
In view of \eqref{1.3n},  for any $\psi \in L^2(0,T; V) $ we deduce that 
\begin{align}\label{galerkin-eti-2}
\begin{split}
 \left|\int_0^T \langle b \partial_t\phi_n ,  \psi\rangle\right|& \le \left| \int_0^T A( \phi_n, \psi) dt\right|+ \left|\int_0^T B(\mathbf{u}_n,  \phi_n, Q_n \psi) d t \right|\\
 &\leq C   \int_0^T \int_\Omega \big|bD \na \phi_n|\,| \na \psi \big|dt +   \int_0^T \int_\Omega \big|  \mathbf{u}_n\cdot \na (Q_n \psi) \phi_n  \big| d\bx  d t  \\
 &\leq C  \left(\| \phi_n\|_{L^2(0,T; V)} +  \|\textbf{u}_n\|_{L^4(0,T; L^4(\Omega))}  \| \phi_n\|_{L^4(0,T; L^4)} \right) \| \psi\|_{L^2(0,T; V)}\\
  &\leq C  \left(\| \phi_n\|_{L^2(0,T; V)} (1+  \|\phi_n\|_{L^\infty(0,T; L^2(\Omega))})  \right)  \| \phi\|_{L^2(0,T; V)}\| \psi\|_{L^2(0,T; V)},
  \end{split}
\end{align} 
where we have used  the interpolation inequality 
$\|\varphi\|^4_{L^4}\le C \|\varphi\|_H^2\|\varphi\|_V^2$. 
This implies that $\partial_t \phi_n \in L^2(0,T; V^*)$.
Since  $\phi_n$ is bounded uniformly (in $n$) in $L^{\infty}\left(0, T ; L^2(\Omega)\right) \cap L^2\left(0, T ; V\right)$,  and $\pa_t\phi_n$ is bounded uniformly (in $n$) in $  L^2(0,T; V^*)$. Standard Sobolev imbedding implies that  $\phi_n$ is bounded uniformly (in $n$) in $L^{\infty}\left(0, T ; L^2(\Omega)\right) \cap L^4 (0, T ; L^4(\Omega))$.
By the Aubin-Simon type compactness results, there exists a function $\phi \in L^2(0,T;H)$ such that $\phi_n\longrightarrow \phi ~~\text{strongly in } ~L^2(0,T;H)$.
Thus, up to subsequences,
\begin{align}\label{con-phin}
\begin{split}
 &\phi_n\longrightarrow \phi ~~\text{strongly in } ~L^2(0,T;H), \text{ and weakly$^*$ ~in} ~L^\infty(0,T;H), \\
 &\phi_n\longrightarrow \phi ~~\text{weakly~in} ~L^2(0,T; V)  \text{ and } L^4(0,T; L^4(\Omega)), \\
    &  b\frac{\partial \phi_n}{\partial t}\longrightarrow  b\partial_t\phi ~~\text{weakly$^*$ ~in}~ L^2(0,T;V^*).
    \end{split}
 \end{align}
 \ignore{
 The strong convergence of $\phi_n$ in $L^2(0,T;H)$ together with the uniform bound in $L^2(0,T; V)$ and interpolation inequality implies
 \begin{equation}
  \phi_n\longrightarrow \phi ~~\text{strongly in } ~L^4(0,T;L^4).
 \end{equation}
 }

 Thanks to \eqref{p-phi} and \eqref{u-phi-lp}, $P_n$ is uniformly (in $n$) bounded in $L^\infty(0,T;  W^{1,2}(\Omega)) \cap L^4 (0, T ; W^{1,4}(\Omega))$, and  $\textbf{u}_n$ is also uniformly (in $n$) bounded in   $L^\infty(0,T; \mathbf{H}) \cap L^4 (0, T ; L^4(\Omega))$. Consequently,  there exist functions  $P\in L^\infty(0,T;  W^{1,2}(\Omega)) \cap L^4 (0, T ; W^{1,4}(\Omega))$, and $  \textbf{u}\in L^4(0,T; L^4(\Omega))\cap   L^\infty(0,T; \mathbf{H}) $ such that, up to subsequences,
 \begin{align}\label{con-un}
 \begin{split}
 &P_n \rightarrow P ~\text{  weakly in  } L^4(0,T; W^{1,4}(\Omega)), ~\text{  and weakly$*$ in }~ L^\infty(0,T; W^{1,2}(\Omega) ),\\
 &\textbf{u}_n \rightarrow \textbf{u} ~\text{  weakly in  } L^4(0,T; L^4(\Omega)), ~\text{  and weakly$*$ in }~ L^\infty(0,T; \mathbf{H}).
 \end{split}
\end{align}
Passing to the limit in $n$ in \eqref{1.1n} and \eqref{1.2n},   \eqref{1.1} and \eqref{1.2'} follow immediately. By standard Sobolev imbedding, $ \phi \in L^2(0,T; L^p(\Omega)) $ for any $1\leq p<\infty$. Thanks to \eqref{p-phi} and \eqref{u-phi-lp}, we know that $\mathbf{u} \in L^2\left(0, T ; L^p(\Omega)\right)$ and $P \in  L^2\left(0, T ; W^{1,p}(\Omega)\right) $, where $2 \le p <\infty$.

On the other hand, for any $\psi \in  L^2(0,T; V)\cap L^\infty(0,T; H)$
\begin{align}\label{con-eq-phi-1}
\int_0^T \langle b\frac{\partial \phi_n}{\partial t},  \psi \rangle dt + \int_0^T\int_\Omega b D \nabla \phi_n   \na \psi d\bx  dt+   \int_0^T\int_\Omega Q_n(\mathbf{u_n} \cdot \nabla \phi_n) \psi  =0.
\end{align}
By \eqref{con-phin},
\begin{align}\label{con-eq-phi-2}
\begin{split}
\int_0^T \langle b\frac{\partial \phi_n}{\partial t},  \psi \rangle dt \longrightarrow  \int_0^T \langle b\frac{\partial \phi }{\partial t},  \psi \rangle dt,\\
\int_0^T\int_\Omega b D \nabla \phi_n   \na \psi d\bx dt   \longrightarrow \int_0^T\int_\Omega b D \nabla \phi    \na \psi d\bx dt .
\end{split}
\end{align}

For the third term in \eqref{con-eq-phi-1}, we note that
\begin{align}\label{con-eq-phi-3}
\begin{split}
 \int_0^T\int_\Omega Q_n(\mathbf{u_n} \cdot \nabla \phi_n) \psi   & =  \int_0^T\int_\Omega  \mathbf{u_n} \cdot \nabla \phi_n  \psi +  \int_0^T\int_\Omega  \mathbf{u_n} \cdot  \nabla  \phi_n  (I-Q_n)\psi\\
 &\doteq \eqref{con-eq-phi-3}_1+ \eqref{con-eq-phi-3}_2.
 \end{split}
\end{align}
Note that
\begin{align} \notag
\begin{split}
 \eqref{con-eq-phi-3}_1= \int_0^T\int_\Omega  (\mathbf{u_n}-\textbf{u}) \cdot \nabla \phi   \psi + \int_0^T\int_\Omega  \mathbf{u_n}  \cdot \nabla (\phi_n-\phi)  \psi + \int_0^T\int_\Omega   \textbf{u}  \cdot \nabla \phi   \psi.
 \end{split}
\end{align}
The weak convergence of $u_n$ in $ L^4(0,T; L^4(\Omega))$ implies that
 \begin{align} \notag
\begin{split}
  \int_0^T\int_\Omega  (\mathbf{u_n}-\textbf{u}) \cdot \nabla \phi   \psi  \longrightarrow 0 \quad \text{ as } n \rightarrow \infty.
 \end{split}
\end{align}
The strong convergence of $\phi_n$ in $ L^2(0,T; L^2)$ and weak convergence in  $ L^2(0,T; V)$, together with simple interpolation, implies that
 \begin{align}\notag
\begin{split}
   \int_0^T\int_\Omega  \mathbf{u_n}  \cdot \nabla (\phi_n-\phi)  \psi= -\int_0^T\int_\Omega  \mathbf{u_n}  (\phi_n-\phi)  \nabla \psi  \longrightarrow 0 \quad \text{ as } n \rightarrow \infty.
 \end{split}
\end{align}
We therefore obtain that
\begin{align} \label{con-eq-phi-5}
\begin{split}
 \eqref{con-eq-phi-3}_1 \rightarrow \int_0^T\int_\Omega   \textbf{u}  \cdot \nabla \phi   \psi.
 \end{split}
\end{align}
By taking $\psi= \sum_{k=1}^m \alpha_k(t) w_k$ with $\alpha_k(t) \in C^1([0,T]; \R)$ and $w_k \in V$, we deduce that
\begin{align}\notag
 \int_0^T\int_\Omega  \mathbf{u_n} \cdot  \nabla  \phi_n  (I-Q_n)\psi= \int_0^T\int_\Omega  \alpha_k(t) \mathbf{u_n} \cdot  \nabla  \phi_n  (I-Q_n) w_k \rightarrow 0,
\end{align}
where we have used the observation $ (I-Q_n) w_k \rightarrow 0 $ in $V$. Since functions of the form $ \sum_{k=1}^m \alpha_k(t) w_k$  with $\alpha_k(t) \in C^1([0,T]; \R)$ and $w_k \in V$ are dense in $ L^2(0,T; V)$, it follows that
\begin{align}\notag
 \int_0^T\int_\Omega  \mathbf{u_n} \cdot  \nabla  \phi_n  (I-Q_n)\psi \rightarrow 0,
\end{align}
for any $ \psi \in  L^2(0,T; V)$. This, combined with \eqref{con-eq-phi-3}, \eqref{con-eq-phi-5},  and also \eqref{con-eq-phi-1} and \eqref{con-eq-phi-2},  gives
\begin{align}\label{con-eq-phi-6}
\int_0^T \langle b\frac{\partial \phi }{\partial t},  \psi \rangle dt + \int_0^T\int_\Omega b D \nabla \phi    \na \psi d\bx dt +   \int_0^T\int_\Omega  \mathbf{u } \cdot \nabla \phi   \psi d\bx dt   =0,
\end{align}
for any $\psi \in  L^2(0,T; V).$  For any $q\in V$ and $t_1, t_2 \in[0,T]$, by taking  $\psi = \chi_{[t_1,t_2]} q$ it follows that  
\begin{align}\left(b \phi(t_2), q\right)-\left(b \phi(t_1), q\right)+\int_{t_1}^{t_2}A( \phi, q) dt+\int_{t_1}^{t_2}B(\mathbf{u},  \phi,q) d t=0, \notag
 \end{align}
which is exactly \eqref{def-eq}. 

\noindent\textbf{The case $d=3$.}
Parallel  to \eqref{galerkin-eti-2},   we have
\begin{align}\label{galerkin-3d-1}
\begin{split}
 \left|\int_0^T \langle b \partial_t\phi_n ,  \psi\rangle\right|& = \left| \int_0^T A( \phi_n, \psi) dt\right|+ \left|\int_0^T B(\mathbf{u}_n,  \phi_n, Q_n \psi) d t \right|\\
 &\leq C  \| \phi_n\|_{L^2(0,T; V)} \| \psi\|_{L^2(0,T; V)}+  \int_0^T \|\textbf{u}_n\|_{L^3(\Omega)}  \| \nabla\phi_n\|_{L^2(\Omega)}   \| \psi\|_{V}\\
  &\leq C  \| \phi_n\|_{L^{2}(0,T; V)} \| \psi\|_{L^4(0,T; V)}+  \int_0^T \|\phi_n\|^{1/2}_{L^2(\Omega)}  \|\nabla \phi_n\|^{3/2}_{L^2(\Omega)}   \| \psi\|_{V}\\
  &\leq C  \left(\| \phi_n\|_{L^2(0,T; V)} +  \|\phi_n\|_{L^\infty(0,T; L^2(\Omega))}  \right)  \| \phi_n\|_{L^2(0,T; V)}\| \psi\|_{L^4(0,T; V)},
  \end{split}
\end{align} which implies that $\partial_t \phi_n$ is uniformly (in $n$) bounded in $L^{4/3}(0,T; V^*)$. By the Aubin-Simon type compactness results, there exists a function $\phi \in L^2(0,T;H)$ such that $\phi_n\longrightarrow \phi ~~\text{strongly in } ~L^2(0,T;H)$.
Thus up to subsequences,
\begin{align}\label{galerkin-3d-2}
\begin{split}
 &\phi_n\longrightarrow \phi ~~\text{strongly in } ~L^2(0,T;H), \text{ and weakly$^*$ ~in} ~L^\infty(0,T;H), \\
 &\phi_n\longrightarrow \phi ~~\text{weakly~in} ~L^2(0,T; V) ~\text{ and }~ L^2(0,T; L^6(\Omega)),\\
    & b \frac{\partial \phi_n}{\partial t}\longrightarrow b\partial_t\phi ~~\text{weakly$^*$ ~in}~ L^{4/3}(0,T;V^*),\\
     &P_n \rightarrow P ~\text{  weakly in  } L^2(0,T; W^{1,6}(\Omega)), ~\text{ weakly$*$ in }~ L^\infty(0,T; W^{1,2}(\Omega) ),\\
 &\textbf{u}_n \rightarrow \textbf{u} ~\text{  weakly in  } L^2(0,T; L^{6}(\Omega)), ~\text{  and weakly$*$ in }~ L^\infty(0,T; \mathbf{H}).
    \end{split}
 \end{align}
Therefore \eqref{con-eq-phi-2} holds for any  $ \psi \in L^4(0,T; V)$. By performing  similar analysis as in \eqref{con-eq-phi-3}--\eqref{con-eq-phi-5}, we can deduce that
\begin{align}\label{galerkin-3d-3}
 \int_0^T\int_\Omega Q_n(\mathbf{u_n} \cdot \nabla \phi_n) \psi  \longrightarrow \int_0^T\int_\Omega  \mathbf{u } \cdot \nabla \phi  \psi
 \end{align}
 for any $ \psi \in L^4(0,T; V)$, which together with \eqref{con-eq-phi-2} gives  \eqref{con-eq-phi-6} and therefore \eqref{def-eq}. 

Finally, we choose a test function $\psi\in C^1([0,T]; V)$ with $\phi(T)=0$ in \eqref{con-eq-phi-6} and the Galerkin approximate problem  \eqref{con-eq-phi-1}.
By passing to the limits in $n$ and using the fact that $\phi_n(0)=Q_n \phi_0\rightarrow \phi(0)$, we obtain that $ \phi(0)=\phi_0(x)$.
The existence of weak solutions is thus proved.
\end{proof}

% S4
\section{Two-dimensional  Regularity and Uniqueness}
\begin{theorem}
In the case $d=2$, the weak solution $(\textbf{u}, \phi, P)$  to problem \eqref{1.1}-\eqref{1.3} subject to the boundary conditions \eqref{IBC_up}, \eqref{IBC}, \eqref{BC}$_1$, \eqref{hBC} together with the periodic conditions in the horizontal direction, and the initial condition \eqref{IC} is unique. If $\phi_0 \in V$,   we have $ \phi_{x}, \textbf{u}_{x}\in L^\infty(0,T; L^2(\Omega) ) \cap L^2(0,T; H^1(\Omega) ),$ and    $ u_2 \in L^\infty(0,T; H^1(\Omega))$ for any $T>0$. Moreover, if the porosity $b$ is a constant  we have \begin{align}
\phi \in L^\infty(0,T; V ) \cap L^2(0,T; W),
\end{align}
\end{theorem}
\begin{proof}
\noindent\textbf{Uniqueness.} Let $(\mathbf{u}, \phi,P_1),(\mathbf{v}, \varphi,P_2)$ be the solutions to problem \eqref{1.1}-\eqref{1.3}  subject to the boundary conditions \eqref{IBC_up}, \eqref{IBC}, \eqref{BC}$_1$, \eqref{hBC} together with the periodic conditions in the horizontal directions, and the initial condition \eqref{IC}. Set $U=\mathbf{u}-\mathbf{v}, \Phi=\phi-\varphi$ and $\hat{P}=P_1-P_2$. We have
\begin{align}
&\operatorname{div}(U) =0, \label{2.24} \\
&U=-\frac{K}{\mu} (\nabla \hat{P}  +\rho_0 \alpha \Phi  g\mathbf{e}_z  ) , \label{2.25} \\
&b \frac{\partial \Phi}{\partial t}+U \cdot \nabla \phi+\mathbf{v} \cdot\nabla \Phi  -\operatorname{div}(b D \nabla \Phi)=0, \label{2.26}
\end{align}
with
\begin{equation}
\begin{aligned}
&U_2\left(x, 0 ; t\right)=U_2\left(x,-H ; t\right)=0, \,\,\,\text { and } \, U\left(0, z ; t\right)=U\left(1, z ; t\right),\\
&\Phi\left(x, 0 ; t\right)=\Phi\left(x, -H ; t\right)=0, \,\,\,\text { and }\, \Phi\left(0, z ; t\right)=\Phi\left(1, z ; t\right).
\end{aligned}
\end{equation}
Multiplying \eqref{2.26} with $\Phi(t)$ and integrating over $\Omega$, we deduce
 \begin{equation}\label{uni1}
\begin{split}
\frac{d}{d t} \int_{\Omega} b|\Phi|^2+\int_{\Omega} b D|\nabla \Phi|^2 &\leq C \left|\int_{\Omega} U \cdot \nabla \Phi \phi \right| \\
& \leq C\|U\|_{L^4(\Omega)}\|\phi\|_{L^{4}(\Omega)}\|\nabla \Phi\|_{L^2(\Omega)}.
\end{split}
\end{equation}
In view of \eqref{u-phi-lp}, we know that
\begin{equation}\label{Phi-U}
\|U(t)\|_{L^4(\Omega)} \leq C\|\Phi(t)\|_{L^4(\Omega)}, \quad \forall t \ge 0.
\end{equation}
Hence, by interpolation,  the RHS of \eqref{uni1} can be bounded as
\begin{equation}\label{uni2}
\begin{split}
\int_{\Omega} U \cdot \nabla \Phi \phi
& \leq C\|U\|_{L^4(\Omega)}\|\phi\|_{L^{4}(\Omega)}\|\nabla \Phi\|_{L^2(\Omega)}\\
&\leq C \|\Phi\|_{L^4(\Omega)}\|\nabla \Phi\|_{L^2(\Omega)}\|\phi\|_{L^{4}(\Omega)}\\
& \leq \|\nabla \Phi\|^{\f{3}{2}}_{L^2(\Omega)}\|\Phi|^{\f{1}{2}}_{L^2(\Omega)}\|\phi\|^{\f{1}{2}}_{L^{2}(\Omega)}\|\nabla\phi\|^{\f{1}{2}}_{L^{2}(\Omega)}\\
& \leq C\|\Phi\|_{L^2(\Omega)}^2\|\phi\|_{L^{2}(\Omega)}^2\|\nabla \phi\|_{L^{2}(\Omega)}^2+\varepsilon_0\|\nabla \Phi\|_{L^2(\Omega)}.
\end{split}
\end{equation}
Here  $\varepsilon_0$ is set as $1/2 \underset{1 \leq j \leq \ell}{\min}\{b_j D_j\}$.  
Inserting  \eqref{uni2} into \eqref{uni1} and using the Gronwall inequality, we obtain that
\begin{equation}
   \|\Phi(t)\|_{L^2(\Omega)}^2 \le  C \|\Phi(0)\|^2_{L^2(\Omega)} \exp \left\{ C \|\phi\|^2_{L^\infty(0,T;H)}\|\phi\|^2_{L^2(0,T;V)} \right\}.
\end{equation}
Since $\Phi(0)=0 $,  we get
 $|\Phi(t)|\equiv 0 , \forall t \in [0, T].$
Furthermore  by \eqref{2.24} and \eqref{2.25} $\hat{P}=0$ up to a constant, and therefore
 $ |U(t)| \equiv 0 , \forall t \in [0, T]$.
The uniqueness is thus proved.

\noindent\textbf{Horizontal regularity} We first consider the horizontal regularity.  Since  $b, D$ are independent of $x$, we differentiate \eqref{1.3} with respect to $x$  to get
\begin{align}
b \frac{\partial (\pa_{x} \phi) }{\partial t}+\mathbf{u} \cdot \nabla \pa_{x} \phi + \mathbf{u}_{x} \cdot \nabla \phi-\operatorname{div}(b D \nabla \pa_{x} \phi )=0  \quad \text{ in } \Omega.\label{1.3difx1}
\end{align}
Note that  $ \pa_{x} \phi \in L^2(0,T; V)$. Multiplying the above equation by $\pa_{x} \phi $ and integrating over $\Omega$, we obtain that
\begin{align}\label{hre01}
&\frac{1}{2} \frac{d}{d t} \int_{\Omega} b |\pa_{x} \phi |^2 +\int_{\Omega} b D|\nabla \pa_{x} \phi |^2
\leq \left|\int_{\Omega}  \mathbf{u}_{x}\cdot \nabla \phi \pa_{x} \phi \right|.
\end{align}
By H\"{o}lder's inequality and by interpolation,
\begin{align} \notag
 \left|\int_{\Omega }  \mathbf{u}_{x}\cdot \nabla \phi \pa_{x} \phi  \right|
 &\leq C \|u_{x}\|_{L^4(\Omega)}\|\pa_{x} \phi \|_{L^{4}(\Omega)}\|\nabla \phi\|_{L^2(\Omega)}   \notag\\
 &\leq C \|\pa_{x} \phi \|^2_{L^4(\Omega)}  \|\nabla\phi\|_{L^2(\Omega)}  \notag\\
&\leq C  \|  \pa_{x} \phi \|^2_{L^2(\Omega)} \|\nabla\phi\|^2_{L^2(\Omega)}+  \varepsilon_0 \| \nabla\pa_{x} \phi \|_{L^{2}(\Omega)}, \notag
\end{align}
 where $\varepsilon_0 =\frac{1}{8}\min_{1\leq j\leq\ell}\{b_jD_j\}$.
 Combing this with \eqref{hre01} yields that
\begin{align}\notag
 \frac{d}{d t} \int_{\Omega} b |\pa_{x} \phi |^2+  \int_{\Omega} b D|\nabla \pa_{x} \phi |^2\leq C  \|  \pa_{x} \phi \|^2_{L^2(\Omega)} \|\nabla\phi\|^2_{L^2(\Omega)}.
\end{align}
As a result,  the Gronwall inequality gives
\begin{align}\label{hre02}
\begin{split}
& \int_{\Omega} b |\pa_{x} \phi (t)|^2\leq   \int_{\Omega} b |\pa_{x} \phi (0)|^2 \exp\left\{  C\| \phi \|^2_{L^2(0,T; V)} \right\},\\
&\int_0^T  \int_{\Omega} b D|\nabla \pa_{x} \phi |^2 \leq C\left\{1+\| \phi(t)\|^2_{L^2(0,T; V)}\exp \{ C \| \phi(t)\|^2_{L^2(0,T; V)}  \} \right\} \int_{\Omega} b |\pa_{x} \phi (0)|^2,
\end{split}
\end{align}
which implies that $ \phi_{x} \in L^\infty(0,T; L^2(\Omega) \cap L^2(0,T; H^1(\Omega))$. By interpolation,  $\phi_{x}  \in L^4(0,T; L^4(\Omega) ).$
In view of \eqref{u-phi-lp}, we have  $ \mathbf{u}_{x} \in L^\infty(0,T; L^2(\Omega) ) \cap L^4(0,T; L^4(\Omega) ).$
Moreover, since $\partial_{z}u_2 = -\partial_{x}u_1$ by the imcompressiblity, if follows that $ u_2 \in L^\infty(0,T; H^1(\Omega))$.

\noindent\textbf{Vertical regularity} We now investigate the regularity in the case that the porosity $b$ is a constant.  Multiply the equation \eqref{1.3} by $ \text{div}(bD \na \phi )$ and integrate over $\Omega$, it follows that
%\begin{align}\label{hre03}
%\begin{split}
%&\frac{1}{2} \frac{d}{d t} \int_{\Omega_j} b D |\pa_{z} \phi |^2+ \int_{\Omega_j} \partial_{x} ( b D  \pa_{x} \phi ) \frac{1}{b} \partial_{z}(bD \pa_{z} \phi )  + \int_{\Omega_j} \frac{1}{b} \big|\partial_{z}(bD \pa_{z} \phi )\big|^2  \\
%&=  - \int_{\Gamma_j^+}b D \pa_{z} \phi  \phi_t d \sigma +\int_{\Gamma_j^-}   b D \pa_{z} \phi  \phi_t d \sigma\\
% &\quad - \int_{\Omega_j}  u_1 \pa_{x} \phi  \frac{1}{b} \pa_{z} ( b D  \pa_{z} \phi)   + \int_{\Omega_j}  u_d  \pa_{z} \phi  \frac{1}{b}\pa_{z}  ( b D  \pa_{z} \phi)  .
%\end{split}\end{align}
%By taking summation and using the interfacial boundary conditions, we get
\begin{align}\label{hre04}
\begin{split}
\frac{d}{d t} \int_{\Omega} b D |\na \phi |^2+  2\int_{\Omega}   \big|\text{div} (bD \na \phi )\big|^2
 &\le  2 \int_{\Omega} |u \cdot \na \phi |  |\text{div}  ( b D  \pa_{z} \phi)|.
  \end{split}\end{align}
In view of \eqref{remark-norm-equi-00} and the interpolation of Sobolev spaces in 2D,
\begin{align}\label{inter-po}
 \|\na \phi \|_{L^4(\Omega)}\leq C  \|\na \phi \|^{1/2}_{L^2(\Omega)}  \|\phi \|^{1/2}_{W}\leq C  \|\na \phi \|^{1/2}_{L^2(\Omega)}  \| \text{div}  ( b D  \na\phi) \|^{1/2}_{L^2(\Omega)}.
\end{align}
By H\"{o}lder's inequality, we deduce from \eqref{hre04} that
\begin{align}\label{hre06}
\begin{split}
 \frac12\frac{d}{d t} \int_{\Omega} b D |\na \phi |^2+   \int_{\Omega}   \big|\text{div} (bD \na \phi )\big|^2
 &\leq C  \|u \|_{L^4(\Omega)}  \|\na \phi \|^{1/2}_{L^2(\Omega)}  \|\text{div}  ( b D  \pa_{z} \phi)\|^{3/2}_{L^2(\Omega)}\\
 &\leq C  \|\phi\|^4_{L^4(\Omega)}  \|\na \phi \|^2_{L^2(\Omega)} +\frac12\|\text{div} (bD \na \phi )\|_H^2,
  \end{split}\end{align}
  which by Gronwall's inequality implies that
  \begin{align}
   \|\na \phi (t) \|^2_{L^2(\Omega)} \leq C  \|\na \phi_0 \|^2_{L^2(\Omega)}\exp\{C\|\phi\|^4_{L^4(0, T; L^4(\Omega))}   \}\leq C   \|\na \phi_0 \|^2_{L^2(\Omega)}\exp\{C\|\phi\|^4_{L^\infty(0, T; L^2(\Omega))} +C \|\phi\|^4_{L^2(0, T; H^1(\Omega))} \} \notag
  \end{align}
 for any $0<t\leq T$. 
 Utilizing this estimate in integrating \eqref{hre06} over $(0,T)$, it yields
 \begin{align}
   \|  \phi\|^2_{L^2(0,T; W)}  &\leq C  \{\|\phi\|^4_{L^\infty(0, T; L^2(\Omega))} + \|\phi\|^4_{L^2(0, T; H^1(\Omega))} \} \notag\\
   &\quad \cdot \exp\{C\|\phi\|^4_{L^\infty(0, T; L^2(\Omega))} +C \|\phi\|^4_{L^2(0, T; H^1(\Omega))} \}  + C\|\phi_0\|_V. \notag
  \end{align}
  The proof is thus complete.
\end{proof}
\begin{remark}
The above tangential a priori estimate is equivalent to taking $\psi=-\phi_{xx}$ in the weak formulation. Such a test function is allowed at the Galerkin level provided we take the eigenfunctions of the second order elliptic operator as the basis. These eigenfunctions are smooth in the horizontal direction(s), and the horizontal derivatives satisfy the same boundary and interfacial conditions as required of the function space.
%Likewise, the vertical a priori estimate can be justified rigorously utilizing Galerkin approximation based on the eigenfunctions of $\mathcal{L}$.
\end{remark}

\begin{remark}
 In the case of constant porosity ($b$ is a constant), we have used $-\mathcal{L}\phi$ as a test function and performed the formal calculations. Such a test function is allowable and all the calculations could be justified rigorously by the Galerkin approximation.

If $b$ is not constant, one formally needs to take  $- \frac{1}{b}\mathcal{L}\phi$ to deduce the desired estimates. But this is generally not allowable as such a function does not belong to $V$. Also one cannot justify the process via Galerkin approximation. Alternative approach that circumvents this difficulty will be reported elsewhere.
\end{remark}

%%%%% S5
\section{Three-dimensional Regularity and Uniqueness}
We consider the regularity and uniqueness of solutions in 3D in the case that the  porosity $b $ is constant  in $\Omega$.
\begin{theorem} \label{thregu-3d}
Assume that the porosity is a constant. Let $\phi_0 \in V$. Problem \eqref{1.1}-\eqref{1.3} subject to the boundary conditions \eqref{IBC_up}, \eqref{IBC}, \eqref{BC}$_1$, \eqref{hBC} together with the periodic conditions in the horizontal directions, and the initial condition \eqref{IC} admits a unique global solution $(\textbf{u},\phi, P)$, such that $\phi \in L^\infty(0,T; V) \cap L^2(0,T; W)$ for any $T>0$.
\end{theorem}
\begin{proof}
\textbf{Local regularity.}
We just perform the formal calculations, and one could resort to the Galerkin approximation for the rigorous proof.
Let $\phi_0 \in V$. We first prove the local regularity, i.e.,  there exists $T^*>0$ such that the solution $\phi \in L^\infty(0,T^*; V) \cap L^2(0,T^*; W)$.

Multiplying the equation \eqref{1.3} by $ -\text{div}(bD \na \phi)$  and integrating over $\Omega$,  it follows that
\begin{align}\label{regu-3d-00}
\begin{split}
 \frac{1}{2} \frac{d}{d t} \int_{\Omega} b^2 D |\na \phi |^2+ \int_{\Omega}  \big|\text{div} (bD \na  \phi )\big|^2
 &\leq \left| \int_{\Omega}\textbf{ u}\cdot \na  \phi   \,  \text{div} ( b D  \na \phi)\right|.   \end{split}
 \end{align}
By \eqref{u-phi-lp} and the standard Sobolev imbedding,
 $ \|\textbf{u}\|_{L^6(\Omega)}  \leq C  \|  \phi\|_{L^6(\Omega)} \leq C  \|\na \phi\|_{L^2(\Omega)}$.  In view of \eqref{remark-imdedding-1}, we have
 \begin{align}
 \|\na  \phi \|_{L^3(\Omega)}
 \leq   \|\na  \phi \|^{1/2}_{L^2(\Omega)}    \|\text{div}( b D  \na \phi)\|^{1/2}_{L^2(\Omega)}.  \notag
\end{align}
By H\"{o}lder's inequality,
\begin{align} \label{regu-3d-0}
\begin{split}
  \frac{1}{2} \frac{d}{d t} \int_{\Omega} b^2 D |\na \phi |^2+ \int_{\Omega}  \big|\text{div} (bD \na  \phi )\big|^2
  &\leq   \|\phi \|_{L^6(\Omega)}  \|\na \phi \|^{1/2}_{L^2(\Omega)}  \|\text{div}  ( b D  \pa_{z} \phi)\|^{3/2}_{L^2(\Omega)}\\
 &\leq  C \|\na \phi\|^6_{L^2(\Omega)} +\va_0 \|\text{div}  ( b D  \pa_{z} \phi)\|^2_{L^2(\Omega)},
 \end{split}
 \end{align}
 which by integration implies that
 \begin{align}
   \int_{\Omega} b D |\na \phi |^2(t)   \leq   \frac{\int_{\Omega} b D |\na \phi_0 |^2   }{1-Ct\int_{\Omega} b D |\na \phi_0 |^2    }   \notag
 \end{align}
 for $0\leq t<  (C\int_{\Omega} b D |\na \phi_0 |^2  )^{-1}.$ Let $T*= (3C\int_{\Omega} b D |\na \phi_0 |^2   )^{-1}$. We have
 \begin{align}\label{regu-3d-2}
   \int_{\Omega} b D |\na \phi(t) |^2 \leq 3 \int_{\Omega} b D |\na \phi_0 |^2 \quad \text{ for any } t\in [0,T*].
 \end{align}
By integrating \eqref{regu-3d-0} over $[0,T^*]$ and using \eqref{regu-3d-2}, we obtain that
\begin{align}
 \int_0^{T^*}\int_{\Omega}  \big|\text{div} (bD \na  \phi )\big|^2   \leq C (T^* + 1)\int_{\Omega} b D |\na \phi_0 |^2  .
\end{align}

\textbf{Global regularity.}
We next prove that $T^*=\infty$ by contradiction. If $T^*<\infty$, then we have $ \limsup_{t\rightarrow T^*} \|\phi(t)\|_V =\infty$. We shall prove that this is  impossible.
 Thanks to the imbedding result proved in section 1,
\begin{align} \|\phi\|_{L^\infty(\Omega)} %\leq C \|\phi\|_{W^{1,4}(\Omega)}
 \leq C \| \phi\|^{1/2}_{H^1(\Omega)} \|\phi\|^{1/2}_{W}, \notag
\end{align}
we have
\begin{align} \|\phi\|_{L^4(0,T; L^\infty(\Omega))} \leq C \|\phi\|^{1/2}_{L^\infty(0,T;H^1(\Omega))}  \|\phi\|^{1/2}_{L^2(0,T;W)}, \notag\end{align}
and therefore $|\phi|^2 \phi  \in  L^2(0,T; V)$ if $ \phi \in L^\infty(0,T; V) \cap  L^2(0,T; W)$.
Multiplying \eqref{1.3} by $|\phi|^2 \phi  $ and integrating over $\Omega$, we obtain that
\begin{align}\label{regu-3d-3}
 \frac{1}{4} \frac{d}{d t} \int_{\Omega}  b | \phi |^4+ 3 \int_{\Omega} bD  | \na  \phi  |^2 |\phi|^2
 \leq 0,
 \end{align}
 which implies that
 \begin{align}\label{regu-3d-4}
  \|\phi(t)\|_{L^4(\Omega)} \leq \|\phi_0\|_{L^4(\Omega)}  \quad \text{ for any  }   0<t<T.
 \end{align}
  By H\"{o}lder's inequality, we deduce from \eqref{regu-3d-00} that
\begin{align}  \label{regu-3d-5}
\begin{split}
  &\frac{1}{2} \frac{d}{d t} \int_{\Omega} b^2 D |\na \phi |^2+ \int_{\Omega}  \big|\text{div} (bD \na  \phi )\big|^2\\
  &\leq    \|\textbf{u }\|_{L^4(\Omega)}  \|\na \phi \|_{L^4(\Omega)}  \|\text{div}  ( b D  \pa_{z} \phi)\|_{L^2(\Omega)}\\
  &\leq   C \|\phi\|_{L^4(\Omega)}  \|\na \phi \|^{1/2}_{L^2(\Omega)}  \|\text{div}  ( b D  \nabla \phi)\|^{3/2}_{L^2(\Omega)}\\
 &\leq  C \|\phi\|^8_{L^4(\Omega)}  \|\na \phi \|^2_{L^2(\Omega)} +\va_0 \|\text{div}  ( b D  \nabla \phi)\|^2_{L^2(\Omega)},
 \end{split}
 \end{align}
 where we have used \eqref{u-phi-lp} and the interpolation
  \begin{align}
 \|\na  \phi \|_{L^4(\Omega)}
 \leq   \|\na  \phi \|^{1/2}_{L^2(\Omega)}    \|\text{div}( b D  \na \phi)\|^{1/2}_{L^2(\Omega)} \notag
\end{align} in the second step.
This by Gronwall's inequality and \eqref{regu-3d-4}  implies that
\begin{align}
    \|\na \phi (t)\|^2_{L^2(\Omega)}  \leq  \|\na \phi_0 \|^2_{L^2(\Omega)}  \exp\{ C \|\phi_0\|^8_{L^4(\Omega)} T   \}  \quad \text{ for any  }   0<t<T. \notag
 \end{align}
  It follows that $ \limsup_{t\rightarrow T^*} \|\phi(t)\|_V  \leq  \|\na \phi_0 \|^2_{L^2(\Omega)}  \exp\{ C \|\phi_0\|^8_{L^4(\Omega)} T^*   \}$,  which is a contradiction. We therefore obtain that $T^*=\infty.$

\textbf{Uniqueness.}  We finally prove the uniqueness of the regular solution.   Let $(\mathbf{u}, \phi, P_1)$, $ (\mathbf{v}, \varphi, P_2)$ be two regular solutions to problem \eqref{1.1}-\eqref{1.3}  subject to the boundary conditions \eqref{IBC_up}, \eqref{IBC}, \eqref{BC}$_1$, \eqref{hBC} together with the periodic conditions in the horizontal directions, and the initial condition \eqref{IC}. Set $U=\mathbf{u}-\mathbf{v}, \Phi=\phi-\varphi$ and $\hat{P}=P_1-P_2$. We have
\begin{align}
&\operatorname{div}(U) =0, \label{2.24-3d} \\
&U=-\frac{K}{\mu} (\nabla \hat{P}  +\rho_0 \alpha \Phi g\mathbf{e}_z  ) , \label{2.25-3d} \\
&b \frac{\partial \Phi}{\partial t}+U \cdot \nabla \phi+\mathbf{v} \cdot\nabla \Phi  -\operatorname{div}(b D \nabla \Phi)=0, \label{2.26-3d}
\end{align}
where $(U, \Phi, \hat{P})$ satisfies  \eqref{IBC_up}, \eqref{IBC}, \eqref{BC}$_1$, \eqref{hBC} together with the periodic conditions in the horizontal directions, and  $ \Phi(0)=0.$

Multiplying \eqref{2.26} with $\Phi$ and integrating over $\Omega$, it yields
 \begin{equation}\label{uni1-3d}
\begin{split}
\frac{d}{d t} \int_{\Omega} b|\Phi|^2+\int_{\Omega} b D|\nabla \Phi|^2 &\leq C \left|\int_{\Omega} U \cdot \nabla \Phi \phi \right| \\
& \leq C\|U\|_{L^4(\Omega)}\|\phi\|_{L^{4}(\Omega)}\|\nabla \Phi\|_{L^2(\Omega)}.
\end{split}
\end{equation}
Note that
\begin{equation}
\|U(t)\|_{L^4(\Omega)} \leq C\|\Phi(t)\|_{L^4(\Omega)}, \quad \forall t \ge 0.  \notag
\end{equation}
By H\"{o}lder's inequality and interpolation,
\begin{equation}\label{uni2-3d}
\begin{split}
\left|\int_{\Omega} U \cdot \nabla \Phi \phi \right|
& \leq C\|U\|_{L^4(\Omega)}\|\phi\|_{L^{4}(\Omega)}\|\nabla \Phi\|_{L^2(\Omega)}\\
&\leq C \|\Phi\|_{L^4(\Omega)}\|\nabla \Phi\|_{L^2(\Omega)}\|\phi\|_{L^{4}(\Omega)}\\
& \leq \|\nabla \Phi\|^{\f{7}{4}}_{L^2(\Omega)}\|\Phi|^{\f{1}{4}}_{L^2(\Omega)}\|\phi\|_{L^4(\Omega)} \\
& \leq C\|\Phi\|_{L^2(\Omega)}^2\|\phi\|_{L^4(\Omega)}^8 +\varepsilon_0\|\nabla \Phi\|^2_{L^2(\Omega)}.
\end{split}
\end{equation}
Here  $\varepsilon_0$ is set as $1/2 \underset{1 \leq j \leq \ell}{\min}\{b_j D_j\}$.  
Inserting  \eqref{uni2-3d} into \eqref{uni1-3d} and using the Gronwall inequality, we obtain that
\begin{equation}
  \|\Phi(t)\|^2_{L^2(\Omega)} \le   \|\Phi(0)\|^2_{L^2(\Omega)} \exp \left\{ C  T   \|\phi_0\|^8_{L^4(\Omega)} \right\},
\end{equation}
which implies that $\Phi(t)=0 $ for any $t>0$ since $\Phi(0)=0 $. In view of  \eqref{2.24-3d} and \eqref{2.25-3d}, we derive the uniqueness of $P$ (up to a constant) and $\textbf{u}$.
The uniqueness is thus proved.
\end{proof}

\begin{remark}\label{remark-max}
The global existence of regular solutions (global regularity) could also be proved by the maximum principle \cite{temam1988ams, ly1999jns, oliver2000gevrey} utilizing the imbedding that we proved in Lemma 2.3.
 \end{remark}

\section{Conclusion and Remarks}

We have derived the well-posedness of convection in layered porous media model under appropriate assumptions. This is the first mathematical result of its kind on this physically important model. Unlike the classical homogenous material case, the solution in no longer smooth, but only piecewise smooth.

There are a few important problems remain to be addressed. First, the physically important case of piecewise constant porosity need to be studied. Second, the jump discontinuity of the material properties should be an idealization of rip change of these properties over a small interval. This resembles the relationship between sharp and diffuse interface limit as the material properties converge from continuous function with rapid transition region to piecewise constant functions. It is thus of interest to investigate the 'sharp interface limit' of these models.
Third,  the rate of transport of mass in the vertical direction, an analogy of the Nusselt number, is of importance. in applications.  We are particularly interested in the impact of the layered structure (the disparity in the material parameters). 
These and other physically important problems are the subject of subsequent works.

%{\color{red} Wang will finish this.}

%
%\section{Acknowledgements}

%{\color{red}Update the ref}

\iffalse
\fi

%\section*{References}
\bibliographystyle{amsalpha}
\bibliography{layer}

\end{document}